\documentclass[11pt,a4paper,reqno]{amsart}
\usepackage[T1]{fontenc}
\usepackage{lmodern}
\usepackage{amsmath,amssymb,mathtools}
\usepackage[margin=29mm]{geometry}
\usepackage{microtype}
\usepackage{xurl}
\usepackage[hidelinks]{hyperref}
\hypersetup{pdftitle={Packing arithmetic progressions: sharp asymptotics and counterexamples},pdfauthor={Jianfeng Hou, Siyue Liu, Hongbin Zhao},pdfsubject={Resolution of Conjectures 1--6 of Alon, Debski, Grytczuk and Przybylo}}
\newtheorem{theorem}{Theorem}[section]
\newtheorem{corollary}[theorem]{Corollary}
\newtheorem{lemma}[theorem]{Lemma}

\newtheorem{conjecture}[theorem]{Conjecture}
\numberwithin{equation}{section}
\newcommand{\Z}{\mathbb Z}
\newcommand{\N}{\mathbb N}
\newcommand{\cP}{\mathcal P}
\newcommand{\cF}{\mathcal F}
\newcommand{\eps}{\varepsilon}
\DeclareMathOperator{\lcm}{lcm}

\DeclarePairedDelimiter{\ceil}{\lceil}{\rceil}
\DeclarePairedDelimiter{\norm}{\lVert}{\rVert}
\title[Arithmetic Progression Packings]{Simultaneous Residue-Class Selection in Prescribed-Difference Packings}
\author[J. Hou]{Jianfeng Hou}
\author[S. Liu]{Siyue Liu}
\author[H. Zhao]{Hongbin Zhao}
\address{Center for Discrete Mathematics, Fuzhou University, Fuzhou, Fujian 350108, China}
\email[Jianfeng Hou]{jfhou@fzu.edu.cn}
\email[Siyue Liu]{siyue\_ll@163.com}
\email[Hongbin Zhao]{hbzhao2024@163.com}
\date{}
\begin{document}
\begin{abstract}

A family $\mathcal F$ of finite subsets of $\mathbb Z$ is packed into $[N]$ if suitable integer translates of its members are pairwise disjoint subsets of $[N]$. We study two prescribed-difference packing problems of Alon, Dębski, Grytczuk and Przybyło for the arithmetic progressions $A_d=\{d,2d,\ldots,\lfloor n/d\rfloor d\}$ and $B_d=\{d,2d,\ldots,nd\}$. Let $m(n)$ and $M(n)$ denote the corresponding minimum packing lengths,
with subscripts indicating restrictions on the admissible differences,
and let $\mathcal P(x)$ denote the set of primes at most $x$.

The key ingredient is a residue-class selection scheme that encodes
pairwise intersection constraints by cyclic intervals. A
lattice-covering argument yields a simultaneous admissible choice,
permitting substantial overlap of the containing intervals and
providing the sharp upper bounds needed for the bounded-diameter
family and the prime-difference equal-cardinality family.

For the bounded-diameter family, we prove
$m(n)\sim m_{\mathcal P(\sqrt n)}(n)\sim 4n^{3/2}/(3\log n)$. For the equal-cardinality family with prime differences, we prove $M_{\mathcal P(n)}(n)\sim n^3/(6\log n)$. For the full equal-cardinality family, we show $M(n)\ge \left(\frac{19}{108}-o(1)\right)n^3/\log n$. Together with corresponding estimates for restricted ranges of differences, these results prove several conjectures of Alon--D\k{e}bski--Grytczuk--Przyby\l{}o and disprove others.
\end{abstract}
\maketitle

\section{Introduction}

For $N\in\N$, write $[N]=\{1,\ldots,N\}$. A finite family
$(C_d)_{d\in D}$ of finite subsets of $\mathbb Z$ is \emph{packed}
into $[N]$ if there are integers $\sigma_d$ such that
$\sigma_d+C_d\subseteq[N]$ for every $d\in D$ and the shifted sets are
pairwise disjoint. Alon, D\k{e}bski, Grytczuk and
Przyby\l{}o~\cite{ADGP} introduced two prescribed-difference variants
of this problem and initiated the systematic study of their extremal
packing lengths.

Fix $n\in\N$. For $d\in[n]$, put
$A_d=\{jd:1\le j\le\lfloor n/d\rfloor\}$, and for $d\in\N$, put
$B_d=\{jd:1\le j\le n\}$. For $D\subseteq[n]$, let $m_D(n)$ be the
least $N$ for which $(A_d)_{d\in D}$ can be packed into $[N]$; write
$m_k(n)=m_{[k]}(n)$ and $m(n)=m_n(n)$. For finite
$D\subseteq\N$, let $M_D(n)$ be the least $N$ for which
$(B_d)_{d\in D}$ can be packed into $[N]$; write
$M_k(n)=M_{[k]}(n)$ and $M(n)=M_n(n)$. We also put
$\cP(x)=\{p\le x:p\text{ prime}\}$ and $\pi(x)=|\cP(x)|$.

The two models arise from packing questions governed by distinct
distance constraints. One source is Cipra's barricade puzzle,
discussed by Guy~\cite{Guy}, which asks for a largest family of
permutations of $[n]$ whose sets of proper partial sums are pairwise
disjoint. The elementary upper bound is $\lfloor n/2\rfloor+1$,
and Alon et al.~\cite{ADGP} conjectured that this bound is sharp.

A more direct source is Johnson's problem of \emph{perfect rhythmic
tilings}, introduced in a 2004 lecture and developed in his later
account~\cite{Johnson2004,Johnson2011}. It asks for a partition of
$[nk]$ into $k$ arithmetic progressions of cardinality $n$ with
distinct positive differences. The prescribed-difference version
$M_k(n)$ requires the differences to be exactly $[k]$. The original
exact tiling problem remains open even for $n=3$~\cite{ADGP}.
For $n=2$, the classical results of Skolem~\cite{Skolem} and
O'Keefe~\cite{OKeefe} give
$M_k(2)=2k$ for $k\equiv0,1\pmod4$ and
$M_k(2)=2k+1$ for $k\equiv2,3\pmod4$.

We first consider the bounded-diameter family $(A_d)$. Alon et
al.~\cite{ADGP} proved
$(4/3-o(1))n^{3/2}/\log n\le m(n)\le
(5/3+o(1))n^{3/2}/\log n$; their lower bound already holds for the
prime differences $\cP(\sqrt n)$. They further proved
$m_k(n)\sim kn/\log k$ when $k\to\infty$ and
$k=o(\sqrt n)$. Motivated by these results, they proposed the
following asymptotics.

\begin{conjecture}[Alon et al.~\cite{ADGP}]
\label{conj:bounded-main}
As $n\to\infty$,
$m(n)\sim m_{\cP(\sqrt n)}(n)
\sim4n^{3/2}/(3\log n)$. Moreover, if
$k=k(n)<\sqrt n$ and $k=\Omega(\sqrt n)$, then
$m_k(n)\sim(1-k^2/(3n))kn/\log k$.
\end{conjecture}

Our first theorem proves all assertions in
Conjecture~\ref{conj:bounded-main} and gives a uniform asymptotic
formula throughout the transition range below $\sqrt n$.

\begin{theorem}\label{thm:bounded}
As $n\to\infty$,
\begin{equation}\label{eq:full-bounded}
m(n)\sim m_{\cP(\sqrt n)}(n)
\sim\frac{4n^{3/2}}{3\log n}.
\end{equation}
For every fixed $c\in(0,1)$, uniformly over integers
$c\sqrt n\le k<\sqrt n$,
\begin{equation}\label{eq:transition}
m_k(n)=
\left(1-\frac{k^2}{3n}+o_c(1)\right)\frac{kn}{\log k}.
\end{equation}
\end{theorem}

\begin{corollary}\label{cor:bounded}
Uniformly for $\sqrt n\le k\le n$,
$m_k(n)\sim4n^{3/2}/(3\log n)$.
\end{corollary}

Thus the leading asymptotic stabilizes at the threshold
$k=\sqrt n$: once $k\ge\sqrt n$, enlarging the set of admissible
differences does not change the leading term.

We next consider the equal-cardinality family $(B_d)$. Alon et
al.~\cite{ADGP} proved
$(1/6-o(1))n^3/\log n\le M(n)\le
(0.526+o(1))n^3/\log n$. They proposed the following asymptotic
behavior for the full family, the prime-difference subfamily, and the
restricted families $(B_d)_{d\le k}$.

\begin{conjecture}[Alon et al.~\cite{ADGP}]
\label{conj:equal-main}
As $n\to\infty$,
$M(n)\sim M_{\cP(n)}(n)\sim n^3/(6\log n)$.
Moreover, suppose that $n,k\to\infty$ with $k<n$. If $k=o(n)$, then
\[
M_k(n)\sim\frac{nk^2}{2\log k},
\]
while otherwise
\[
M_k(n)\sim
\frac{nk^2}{2\log k}-\frac{k^3}{3\log k}.
\]
In particular, along any sequence $k/n\to\lambda\in(0,1)$, the latter
relation becomes
\[
M_k(n)\sim
\left(\frac{\lambda^2}{2}-\frac{\lambda^3}{3}\right)
\frac{n^3}{\log n}.
\]
\end{conjecture}

For each fixed $n$ and all sufficiently large $k$, Alon et
al.~\cite{ADGP} proved $nk\le M_k(n)\le3nk$. Subsequently,
Mao, Wang, Wei and Yang~\cite{MWWY} established $M_k(n)\sim nk$,
thereby settling Conjecture~7 in~\cite{ADGP} in this separate regime.

Our second theorem confirms the prime-difference asymptotic in
Conjecture~\ref{conj:equal-main}, but gives stronger lower bounds
when $k=o(n)$ and along every sequence
$k/n\to\lambda\in(0,1)$.

\begin{theorem}\label{thm:equal}
As $n\to\infty$, $M_{\cP(n)}(n)\sim n^3/(6\log n)$.
If $k\to\infty$ and $k=o(n)$, then
$M_k(n)\ge(9/16-o(1))nk^2/\log k$.
If $k/n\to\lambda\in(0,1)$ and $a=\min\{\lambda/2,1/3\}$, then
\[
M_k(n)\ge
\left(\frac{\lambda^2}{2}-\frac{\lambda^3}{3}
+\frac{a^2}{4}-\frac{a^3}{2}-o(1)\right)
\frac{n^3}{\log n}.
\]
\end{theorem}

The additional term equals $\lambda^2(1-\lambda)/16$ for
$0<\lambda\le2/3$ and $1/108$ for $2/3\le\lambda<1$.
Consequently, the two predictions for $M_k(n)$ in
Conjecture~\ref{conj:equal-main} are false. The following consequence
also disproves the asserted asymptotic for the full family.

\begin{corollary}\label{cor:equal}
As $n\to\infty$,
$M(n)\ge(19/108-o(1))n^3/\log n$.
\end{corollary}

Thus the prime-difference asymptotic does not extend to the full
equal-cardinality family: already the pairs $p,2p$ produce a positive
correction at the leading-order scale. This contrasts sharply with
the bounded-diameter problem, where the prime-difference subfamily
already determines the full leading asymptotic.

% These packing problems are also related to classical union problems.
% The multiplication-table problem concerns the order of
% $|\bigcup_{d\le n}B_d|$~\cite{Erdos,Ford}, while the arithmetic
% Kakeya problem asks to minimize
% $|\bigcup_{d\le k}(s_d+B_d)|$~\cite{KatzTao,Ruzsa,GreenRuzsa}.
% Our setting differs in requiring the translated copies to be pairwise
% disjoint and in minimizing the length of a containing interval.
These packing problems are also related to classical union problems
in additive and multiplicative combinatorics. The multiplication-table
problem concerns the order of
$|\bigcup_{d\le n}B_d|$~\cite{Erdos,Ford}, while the arithmetic
Kakeya problem asks to minimize
$|\bigcup_{d\le k}(s_d+B_d)|$~\cite{GreenRuzsa,KatzTao,Ruzsa}.
Our setting differs in requiring the translated copies to be pairwise
disjoint and in minimizing the length of a containing interval.

We briefly indicate the proof strategy. The matching upper bounds for
the bounded-diameter family and the prime-difference equal-cardinality
family use a common residue-class selection scheme. Pairwise
intersection constraints are encoded by cyclic intervals, and the
Chinese remainder theorem together with a lattice-covering estimate
yields a simultaneous admissible choice after grouping primes into
strongly regular tuples. In the bounded-diameter problem, the same
construction also accommodates small multiples of the primes.

Alon et al.~\cite{ADGP} had already established a lower bound with
leading constant $4/3$ for the bounded-diameter problem and one with
leading constant $1/6$ for the full equal-cardinality family; the
former is already forced by the prime-difference subfamily. Our
stronger lower bounds in the equal-cardinality problem arise from
parity-compatible selections from the pairs $p,2p$.

Section~\ref{sec:prelim} collects the notation and auxiliary results.
Sections~\ref{sec:bounded} and~\ref{sec:equal} give the proofs for the
bounded-diameter and equal-cardinality problems, respectively.

\section{Notation and preliminaries}\label{sec:prelim}

We retain the families $A_d,B_d$ and packing parameters $m_D(n),M_D(n)$
defined in the introduction. Write $\N=\{1,2,\ldots\}$, and let $\Z$
denote the integers. All logarithms are natural, and $p,q$ always denote
primes. Integer intervals are measured by cardinality, so
$|[a,b]|=b-a+1$ for integers $a\le b$; the integer hull of a nonempty
finite set $S\subset\Z$ is $[\min S,\max S]$.
For $x\ge0$, $\cP(x)=\{p\le x:p\text{ prime}\}$ and
$\pi(x)=|\cP(x)|$. We write $f\sim g$ if $f/g\to1$ and
$f\asymp g$ if $f=O(g)$ and $g=O(f)$; subscripts on asymptotic
notation indicate dependence on fixed parameters.
In the main results $n\to\infty$ and $k\le n$; auxiliary parameters
are fixed before the principal variable tends to infinity.

\begin{lemma}[{Shoup~\cite{Shoup}}]\label{lem:crt}
Let $d,e\in\N$ and $a,b\in\Z$. The congruences
$z\equiv a\pmod d$ and $z\equiv b\pmod e$ have a common solution
if and only if $a\equiv b\pmod{\gcd(d,e)}$.
When soluble, their solutions form a single residue class modulo
$L=\lcm(d,e)$; hence every interval of $L$ consecutive integers
contains exactly one solution.
\end{lemma}

\begin{lemma}[{Montgomery--Vaughan~\cite{MV};
Banaszczyk~\cite{Banaszczyk}}]\label{lem:standard}
The prime number theorem and partial summation give
\begin{equation}\label{eq:moments}
\sum_{p\le x}p^j\sim\frac{x^{j+1}}{(j+1)\log x}
\qquad(x\to\infty,\ j=0,1,2).
\end{equation}
The odd-prime sums $S_j(x)=\sum_{p\le x,\ p\text{ odd}}p^j$,
$j=1,2$, have the same asymptotics.

For $v\in\mathbb R^d$, let $\norm v_2=\sqrt{\langle v,v\rangle}$.
For a full-rank lattice $\Lambda\subset\mathbb R^d$, define
$\Lambda^*=\{y\in\mathbb R^d:\langle y,\Lambda\rangle\subseteq\Z\}$,
$\lambda_1(\Lambda)=\min\{\norm v_2:0\ne v\in\Lambda\}$ and
$\rho(\Lambda)=\sup_{x\in\mathbb R^d}\inf_{v\in\Lambda}\norm{x-v}_2$.
The Euclidean transference inequality states that
\begin{equation}\label{eq:transference}
\rho(\Lambda)\lambda_1(\Lambda^*)\le d/2.
\end{equation}
\end{lemma}

For completeness, we record the following finite form of the lower-bound
estimate, derived from the argument of Alon et al.~\cite{ADGP}.

\begin{lemma}
\label{lem:bounded-lower}
For $2\le u\le\sqrt n$,
\begin{equation}\label{eq:bounded-lower}
m_{\cP(u)}(n)\ge n\pi(u)-\sum_{p\le u}p^2-2u\pi(u)-n.
\end{equation}
\end{lemma}

\begin{proof}
Let $(\sigma_p+A_p)_{p\le u}$ be an arbitrary packing in $[N]$ and
let $I_p$ be the integer hull of $\sigma_p+A_p$. Since
$A_p=\{p,2p,\ldots,\lfloor n/p\rfloor p\}$, we have
$I_p=[\sigma_p+p,\sigma_p+\lfloor n/p\rfloor p]$ and
$\sigma_p+A_p=I_p\cap(\sigma_p+p\Z)$. Moreover,
$|I_p|=(\lfloor n/p\rfloor-1)p+1\ge n-2p+1\ge n-2u$.
For distinct primes $p,q$, Lemma~\ref{lem:crt} gives one simultaneous
solution of $z\equiv\sigma_p\pmod p$ and
$z\equiv\sigma_q\pmod q$ in every interval of $pq$ consecutive
integers. Hence $|I_p\cap I_q|\ge pq$ would produce an element of both
$\sigma_p+A_p$ and $\sigma_q+A_q$, contrary to the packing property.
Thus $|I_p\cap I_q|<pq$ for $p\ne q$.

Write the hulls as $I_{p_i}=[a_i,b_i]$, ordered so that
$a_1\le\cdots\le a_t$, where $t=\pi(u)$. Put
$\ell_i=b_i-a_i+1$ and $\Delta_i=a_{i+1}-a_i$ for $i<t$. Since
$a_i\le a_{i+1}$,
\[
 |I_{p_i}\cap I_{p_{i+1}}|
 =\max\{0,\min\{\ell_i-\Delta_i,\ell_{i+1}\}\}
 \ge n-2u-\Delta_i.
\]
Together with the strict overlap bound this implies
$\Delta_i\ge n-2u-p_ip_{i+1}$. Also,
\[
 \sum_{i<t}p_ip_{i+1}
 \le\frac12\sum_{i<t}(p_i^2+p_{i+1}^2)
 \le\sum_{i=1}^t p_i^2.
\]
As every $I_{p_i}$ lies in $[N]$, we have
$N\ge a_t-a_1=\sum_{i<t}\Delta_i$. Therefore
\begin{align*}
N&\ge(t-1)(n-2u)-\sum_{i<t}p_ip_{i+1}\\
&\ge(n-2u)(\pi(u)-1)-\sum_{p\le u}p^2
\ge n\pi(u)-\sum_{p\le u}p^2-2u\pi(u)-n,
\end{align*}
which proves \eqref{eq:bounded-lower}.

We record the two asymptotic forms used later. If $u=\sqrt n$, then
$\log u=\frac12\log n$, and \eqref{eq:moments} gives
\begin{align*}
n\pi(u)&=(2+o(1))\frac{n^{3/2}}{\log n},&
\sum_{p\le u}p^2&=\left(\frac23+o(1)\right)
\frac{n^{3/2}}{\log n},\\
2u\pi(u)&=(4+o(1))\frac n{\log n}
=o\left(\frac{n^{3/2}}{\log n}\right),&
n&=o\left(\frac{n^{3/2}}{\log n}\right).
\end{align*}
Substitution in \eqref{eq:bounded-lower} yields
\[
m_{\cP(\sqrt n)}(n)\ge
\left(2-\frac23-o(1)\right)\frac{n^{3/2}}{\log n}
=\left(\frac43-o(1)\right)\frac{n^{3/2}}{\log n}.
\]
Now fix $c\in(0,1)$ and suppose $c\sqrt n\le k<\sqrt n$. Since
$k\ge c\sqrt n\to\infty$, the error terms in \eqref{eq:moments} are
uniform over this range. With $u=k$ they give
\begin{align*}
n\pi(k)&=(1+o_c(1))\frac{nk}{\log k},\\
\sum_{p\le k}p^2
&=\left(\frac13+o_c(1)\right)\frac{k^3}{\log k}
=\left(\frac{k^2}{3n}+o_c(1)\right)\frac{nk}{\log k}.
\end{align*}
Moreover,
\[
\frac{2k\pi(k)+n}{nk/\log k}
=(2+o_c(1))\frac{k}{n}+\frac{\log k}{k}=o_c(1),
\]
because $k/n\le n^{-1/2}$ and
$\log k/k\le c^{-1}\log n/\sqrt n$. Substitution in
\eqref{eq:bounded-lower} therefore yields, uniformly in $k$,
\begin{equation*}
m_{\cP(k)}(n)\ge
\left(1-\frac{k^2}{3n}-o_c(1)\right)\frac{nk}{\log k}
\qquad(c\sqrt n\le k<\sqrt n). 
\qedhere
\end{equation*}
\end{proof}

\begin{lemma}[{Alon et al.~\cite{ADGP}}]
\label{lem:reduction}
For fixed $0<\delta<1$, put $X=\sqrt n(\log n)^3$ and
$\cF_\delta=\bigcup_{\delta\sqrt n<p\le\sqrt n}
\{mp:1\le m\le\lfloor X/p\rfloor\}$.
Then, as $n\to\infty$,
\begin{equation}\label{eq:remainder}
m_{[n]\setminus\cF_\delta}(n)
\le n\pi(\delta\sqrt n)+o_\delta(n^{3/2}/\log n)
=\bigl(2\delta+o_\delta(1)\bigr)\frac{n^{3/2}}{\log n}.
\end{equation}
\end{lemma}

\begin{lemma}\label{lem:compatible}
For a packing $(\sigma_d+B_d)_{d\in D}$ into $[N]$, we say that
$D$ is phase-compatible with respect to this packing if
$\sigma_d\equiv\sigma_e\pmod{\gcd(d,e)}$ for all $d,e\in D$. If $D\subseteq[n-1]$, then
\begin{equation}\label{eq:compatible}
N\ge\sum_{d\in D}(nd-d^2-d+1).
\end{equation}
\end{lemma}

\begin{proof}
The assertion is immediate for $D=\varnothing$, so assume
$D=\{d_1,\ldots,d_t\}\ne\varnothing$. For $d\in D$, let
$I_d=[\sigma_d+d,\sigma_d+nd]$ be the integer hull of
$\sigma_d+B_d$, so that $|I_d|=(n-1)d+1$.

Order $d_1,\ldots,d_t$ so that the left endpoints of
$I_{d_1},\ldots,I_{d_t}$ are nondecreasing. By phase compatibility
and Lemma~\ref{lem:crt}, the congruences corresponding to $d_i$ and
$d_{i+1}$ have a common residue class modulo
$\lcm(d_i,d_{i+1})$. Hence
$|I_{d_i}\cap I_{d_{i+1}}|<\lcm(d_i,d_{i+1})\le d_id_{i+1}$.
Since $d_i\le n-1$, we also have
$d_id_{i+1}\le(n-1)d_{i+1}<|I_{d_{i+1}}|$.
Thus $I_{d_{i+1}}$ cannot be contained in $I_{d_i}$, so the right
endpoints are also strictly increasing.

Let $r_i$ denote the right endpoint of $I_{d_i}$. Then
$r_{i+1}-r_i\ge |I_{d_{i+1}}|-|I_{d_i}\cap I_{d_{i+1}}|$, and hence
\begin{align*}
N
&\ge r_t-\min I_{d_1}+1 \\
&\ge \sum_{i=1}^t |I_{d_i}|
   -\sum_{i<t}|I_{d_i}\cap I_{d_{i+1}}| \\
&\ge (n-1)\sum_{i=1}^t d_i+t-\sum_{i<t}d_id_{i+1} \\
&\ge (n-1)\sum_{i=1}^t d_i+t-\sum_{i=1}^t d_i^2,
\end{align*}
where the last inequality follows from
$2d_id_{i+1}\le d_i^2+d_{i+1}^2$. This gives
\eqref{eq:compatible}.
\end{proof}

A \emph{cyclic interval} in $\Z/r\Z$ is the image of an integer
interval of length at most $r$. If each $q_i$ is coprime to $r$, call
$(q_1,\ldots,q_d;r)$ \emph{$R$-regular} if
$\sum_i h_iq_i^{-1}\not\equiv0\pmod r$ for every
$h\in\Z^d$ with $0<\norm h_2<R$.
Regularity passes to subtuples by inserting zero coordinates.

\begin{lemma}\label{lem:phase}
The following statements hold.

\textup{(i)} Every integer interval of length at least $td$ contains a
translate of $\{d,2d,\ldots,td\}$ in any prescribed residue class modulo $d$.

\textup{(ii)} Let $p,q$ be coprime, fix a residue $a$ modulo $p$ and an
integer lift $\widetilde a$, and let $J$ be a cyclic interval in
$\Z/(pq)\Z$. The parameters $t\in\Z/q\Z$ satisfying
$\widetilde a+pt\in J$ form a cyclic interval of cardinality at least
$\lfloor|J|/p\rfloor$.

\textup{(iii)} Let $d,r,q_1,\ldots,q_d\in\N$ with
$\gcd(q_i,r)=1$ for every $i$. Suppose $(q_1,\ldots,q_d;r)$ is
$R$-regular and $0<R<r$. If cyclic intervals
$T_i\subseteq\Z/r\Z$ satisfy $|T_i|\ge dr/R+3$, then for any integers
$b_1,\ldots,b_d$ there exist $t_i\in T_i$ and $b\in\Z/r\Z$ such that
$b\equiv b_i+q_it_i\pmod r$ for every $i$.
\end{lemma}

\begin{proof}
For \textup{(i)}, take any $td$ consecutive integers inside the given
interval. Each residue class modulo $d$ occurs exactly $t$ times. The
$t$ representatives of the prescribed class are consecutive terms of
one arithmetic progression with difference $d$, and hence are a
translate of $\{d,2d,\ldots,td\}$.

For \textup{(ii)}, choose an integer interval
$\widetilde J=[A,A+|J|-1]$ whose image modulo $pq$ is $J$. The integers
in $\widetilde J$ congruent to $a$ modulo $p$ form a consecutive segment
of an arithmetic progression of difference $p$, containing at least
$\lfloor|J|/p\rfloor$ terms. Because $p$ is invertible modulo $q$, the
map $t\mapsto\widetilde a+pt\pmod{pq}$
is a bijection from $\Z/q\Z$ onto the residues modulo $pq$ congruent to
$a$ modulo $p$. The parameters corresponding to the above consecutive
terms are consecutive modulo $q$ and therefore form the asserted cyclic
interval.

For \textup{(iii)}, define the full-rank lattice
\[
 \Lambda=\{t\in\Z^d:q_it_i\equiv q_1t_1\pmod r
 \,\text{ for }2\le i\le d\}.
\]
It contains $r\Z^d$. Modulo $r\Z^d$, every element of $\Lambda$ is
uniquely of the form
$c(q_1^{-1},\ldots,q_d^{-1})$, $c\in\Z/r\Z$. Hence every vector of
$\Lambda^*$ is $h/r$ for some $h\in\Z^d$, and such a vector belongs to
$\Lambda^*$ exactly when
$\sum_i h_iq_i^{-1}\equiv0\pmod r$. Thus
\[
\Lambda^*=\left\{\frac hr:h\in\Z^d,\
\sum_{i=1}^d h_iq_i^{-1}\equiv0\pmod r\right\}.
\]
If $0\ne h/r\in\Lambda^*$, regularity gives $\norm h_2\ge R$;
therefore $\lambda_1(\Lambda^*)\ge R/r$. By
\eqref{eq:transference},
$\rho(\Lambda)\le d/(2\lambda_1(\Lambda^*))\le dr/(2R)$.
Choose $v\in\Z^d$ with
$v_i\equiv-b_iq_i^{-1}\pmod r$. Then $b_i+q_iv_i\equiv0\pmod r$
for every $i$, and the set of integer vectors $t$ for which all
$b_i+q_it_i$ are congruent modulo $r$ is precisely $v+\Lambda$.

Lift each $T_i$ to an integer interval $\widetilde T_i$ of cardinality
$|T_i|$, and let $\mathcal B=\prod_i\widetilde T_i\subset\mathbb R^d$.
If $c$ is the centre of this box, then every half-side has length
$(|T_i|-1)/2\ge dr/(2R)+1>\rho(\Lambda)$.
By the definition of the covering radius, some $v+w$, $w\in\Lambda$,
satisfies $\norm{c-(v+w)}_2\le\rho(\Lambda)$. Each coordinate distance
is at most this Euclidean distance, so $v+w\in\mathcal B$. Reducing its
coordinates modulo $r$ gives $t_i\in T_i$, and their common residue
$b_i+q_it_i$ gives the required $b\in\Z/r\Z$.
\end{proof}

Fix an integer $s\ge3$ and $0<\eta<1$, and put
$R(P)=P^{1/(s-1)}/\log P$. We call a tuple $p_1>\cdots>p_s$ of
primes in $(P,(1+\eta)P]$ \emph{strongly regular} if
$(p_j:j\ne i;p_i)$ is $R(P)$-regular for every $i$.

\begin{lemma}\label{lem:regular-blocks}
All but $o(P/\log P)$ primes in $(P,(1+\eta)P]$ can be
partitioned into strongly regular $s$-tuples as $P\to\infty$.
Any family of at most $p$ differences in $[n]$ divisible by $p$ fits
in an interval of length $n$.
Moreover, with $\alpha_{s,\eta}=(s-1)/(s(1+\eta)^2)$, every such tuple
satisfies
$\sum_{i=1}^{s-1}p_ip_{i+1}\ge
\alpha_{s,\eta}\sum_{i=1}^s p_i^2$, and the exceptional primes have
squared sum $o(P^3/\log P)$.
\end{lemma}

\begin{proof}
Divide $(P,(1+\eta)P]$ into $s$ consecutive intervals of equal length
and label them from right to left. By the prime number theorem, each
contains $(\eta/s+o(1))P/\log P$ primes. Removing a total of
$o(P/\log P)$ primes, choose subsets $S_1,\ldots,S_s$ of equal size
$|S_i|=M\sim\eta P/(s\log P)$.
Every tuple in $S_1\times\cdots\times S_s$ is automatically ordered
$p_1>\cdots>p_s$.

We estimate the tuples that are not strongly regular. Fix the modulus
position $i$ and a nonzero vector
$h=(h_j)_{j\ne i}\in\Z^{s-1}$ with $\norm h_2<R=R(P)$. Choose
$\ell\ne i$ with $h_\ell\ne0$. Once the primes $p_j$, $j\ne\ell$,
are fixed, failure of regularity for this $i$ and $h$ requires
$h_\ell p_\ell^{-1}\equiv -\sum_{j\ne i,\ell}h_jp_j^{-1}\pmod{p_i}$.
Because $0<|h_\ell|<R<p_i$, the left coefficient is nonzero modulo
$p_i$; the congruence has either no solution or determines one residue
class for $p_\ell$ modulo $p_i$. The interval $(P,(1+\eta)P]$ has
length $\eta P<p_i$, so this residue class contains at most one member
of $S_\ell$. Thus, for fixed $i,h$, at most $M^{s-1}$ tuples are bad.
There are $O_s(R^{s-1})$ possible $h$ and $s$ choices of $i$, whence
the total number $\mathcal N$ of bad tuples satisfies
\[
 \mathcal N=O_s(R^{s-1}M^{s-1})
 =O_{s,\eta}\left(\frac{M^s}{(\log P)^{s-2}}\right)=o(M^s),
\]
where we used $R^{s-1}=P/(\log P)^{s-1}$ and
$M\asymp_{s,\eta}P/\log P$.

Fix an ordering of $S_1$ and choose independent uniformly random
permutations of $S_2,\ldots,S_s$. The entries in each of the $M$ rows
form a tuple, and different rows use disjoint primes. Each row is
uniform on $S_1\times\cdots\times S_s$ after its first coordinate is
fixed; equivalently, summing over the rows, its expected number of bad
tuples is $M\mathcal N/M^s=\mathcal N/M^{s-1}=o(M)$.
Hence some choice of the permutations has only $o(M)$ bad rows.
Deleting those rows, together with the primes discarded when the sets
$S_i$ were equalized, leaves strongly regular $s$-tuples and only
$o(P/\log P)$ exceptional primes.

For the divisibility assertion, let $E\subseteq[n]$ contain at most
$p$ differences, all divisible by $p$. Assign distinct residues
$c_d\pmod p$ to the members $d\in E$. Since $p\mid d$, choose a
residue $r_d\pmod d$ satisfying $r_d\equiv c_d\pmod p$. Every integer
interval of length $n$ contains at least $\lfloor n/d\rfloor$
consecutive representatives of $r_d\pmod d$; these representatives
form a translate of $A_d$. The translates for distinct $d$ are
disjoint because their elements have distinct residues modulo $p$.
Thus all the corresponding progressions fit in the same length-$n$
interval.

Finally, for a tuple in the stated prime interval,
\[
 \sum_{i=1}^{s-1}p_ip_{i+1}>(s-1)P^2,
 \qquad
 \sum_{i=1}^s p_i^2\le s(1+\eta)^2P^2,
\]
which gives the factor $\alpha_{s,\eta}$. The exceptional set contains
$o(P/\log P)$ primes, each at most $(1+\eta)P$, so its squared sum is
$o(P^3/\log P)$.
\end{proof}

The following construction treats both models, allowing small multiples
of each prime in the bounded-diameter model.

\begin{lemma}\label{lem:blocks}
Let $p_1>\cdots>p_s$ be a strongly regular tuple in $(P,(1+\eta)P]$,
and let $1\le h_i\le H=H(P)=o(R(P))$ be integers.
For all sufficiently large $P$, the family
$\{A_{mp_i}:1\le i\le s,\ 1\le m\le h_i\}$, with $n\ge p_1^2$,
and the family $\{B_{p_i}:1\le i\le s\}$, with $n\ge p_1$,
admit packings in intervals of lengths at most, respectively,
\begin{align*}
&sn-\sum_{i=1}^{s-1}p_ip_{i+1}+O_{s,\eta}(HP^2/R(P)),\\
&n\sum_{i=1}^s p_i-\sum_{i=1}^{s-1}p_ip_{i+1}
+O_{s,\eta}(P^2/R(P)).
\end{align*}
\end{lemma}

\begin{proof}
Put $Q=(1+\eta)P$, $R=R(P)$ and $L=\ceil{(s-1)Q/R}+3$.
For the $A$-family set $D=\ceil{2HQL}$ and $\ell_i=n$;
for the $B$-family set $D=0$ and $\ell_i=np_i$.
In both cases put $\Gamma=D+\ceil{QL}$.
Since $R=P^{1/(s-1)}/\log P$, we have $R=o(P)$ and
$L=O_{s,\eta}(P/R)=o(P)$. The hypothesis $H=o(R)$ implies that
$H,L<p_s$ for all sufficiently large $P$. Furthermore,
\[
 \Gamma=
 \begin{cases}
 O_{s,\eta}(HP^2/R),&\text{for the $A$-family},\\
 O_{s,\eta}(P^2/R),&\text{for the $B$-family},
 \end{cases}
\]
so $\Gamma=o(P^2)$ in both cases.

We first realize the progressions associated with a fixed prime $p_i$.
In the $A$-case, apply Lemma~\ref{lem:phase}\textup{(iii)} with
modulus $p_i$, coefficients $(p_j)_{j\ne i}$, zero offsets, and
coordinate intervals $[1,L]$. Strong regularity supplies the required
$R$-regularity, and
$L\ge (s-1)Q/R+3\ge (s-1)p_i/R+3$. We obtain
$t_{ij}\in[1,L]$ and a residue $\rho_i\pmod{p_i}$ such that
$\rho_i\equiv p_jt_{ij}\pmod{p_i}$ for $j\ne i$.
The residue $\rho_i$ is nonzero, since otherwise $p_i\mid t_{ij}$,
whereas $1\le t_{ij}\le L<p_i$.

Let $a_i\in\Z/p_i\Z$ be chosen later. For $1\le m\le h_i$, prescribe
the residue $a_i+m\rho_i\pmod{p_i}$ for $A_{mp_i}$. These residues are
pairwise distinct because $h_i<p_i$ and $\rho_i\ne0$. Choose a lift of
each prescribed residue modulo $mp_i$. With
$t=\lfloor n/(mp_i)\rfloor$, Lemma~\ref{lem:phase}\textup{(i)} places
a translate of $\{mp_i,2mp_i,\ldots,tmp_i\}=A_{mp_i}$ in that lifted
class inside any interval of length $n$. Hence all progressions
associated with $p_i$ fit disjointly in one length-$n$ interval.

In the $B$-case, prescribe the residue $a_i\pmod{p_i}$ for $B_{p_i}$.
Every interval of $np_i$ consecutive integers contains exactly $n$
representatives of this residue, and they form a translate of $B_{p_i}$.

Set
\[
 J_i=[u_i,u_i+\ell_i-1],\qquad u_1=0,\qquad
 u_{i+1}-u_i=\ell_i-p_ip_{i+1}+2\Gamma.
\]
Both endpoint sequences are strictly increasing. Indeed, the
left-endpoint increment is positive because $\ell_i\ge p_ip_{i+1}$:
use $n\ge p_1^2$ in the $A$-case and $np_i\ge p_ip_{i+1}$ in the
$B$-case. The right-endpoint increment is
$\ell_{i+1}-p_ip_{i+1}+2\Gamma>0$, by the same hypotheses.

For the $A$-family,
$(n-ab)+(n-bc)-(n-ac)=(\sqrt n-a)(\sqrt n-c)
+(\sqrt n-b)(a+c)\ge0$ for $0<a,b,c\le\sqrt n$. For the $B$-family,
$a(n-b)+b(n-c)-a(n-c)=b(n-a)+c(a-b)\ge0$ when
$n\ge a\ge b\ge c>0$. We claim that in either case
$u_j-u_i\ge\ell_i-p_ip_j+2(j-i)\Gamma\qquad(i<j)$.
The case $j=i+1$ is the definition of $u_{i+1}$. If the claim holds
for $(i,j)$, add
$u_{j+1}-u_j=\ell_j-p_jp_{j+1}+2\Gamma$. The claim for $(i,j+1)$
then reduces to
$\ell_j-p_ip_j-p_jp_{j+1}+p_ip_{j+1}\ge0$,
which is the corresponding preceding identity with
$(a,b,c)=(p_i,p_j,p_{j+1})$. This proves the claim by induction.
Since both endpoint sequences increase, a nonempty intersection has
cardinality $|J_i\cap J_j|=u_i+\ell_i-u_j
\le p_ip_j-2(j-i)\Gamma\le p_ip_j-2\Gamma$.

It remains to choose the phases $a_i$. For $i<j$,
Lemma~\ref{lem:crt} gives a unique residue $C_{ij}\pmod{p_ip_j}$ with
$C_{ij}\equiv a_i\pmod{p_i}$ and
$C_{ij}\equiv a_j\pmod{p_j}$. In the $A$-case, a collision between
the copies of $A_{mp_i}$ and $A_{vp_j}$ would satisfy
\[
x\equiv C_{ij}+mp_jt_{ij}+vp_it_{ji}\pmod{p_ip_j},
\qquad 0\le mp_jt_{ij}+vp_it_{ji}\le D.
\]
Indeed, modulo $p_i$ the displayed residue equals
$a_i+mp_jt_{ij}=a_i+m\rho_i$, and modulo $p_j$ it equals
$a_j+vp_it_{ji}=a_j+v\rho_j$. Its offset from $C_{ij}$ is at most
$2HQL\le D$. In the $B$-case, a collision has residue $C_{ij}$ and
$D=0$.

Let $E_{ij}=J_i\cap J_j$. If $E_{ij}\ne\varnothing$, let $F_{ij}$ be
the image modulo $p_ip_j$ of the integer interval
$E_{ij}-\{0,1,\ldots,D\}$. Avoiding $C_{ij}\in F_{ij}$ excludes all
possible collisions. The overlap bound gives
$|F_{ij}|\le |E_{ij}|+D \le p_ip_j-2\Gamma+D<p_ip_j$,
so the complement of $F_{ij}$ is a cyclic interval containing at least
$2\Gamma-D$ residues.

Choose $a_1$ arbitrarily and suppose $a_1,\ldots,a_{j-1}$ have been
chosen. For $i<j$ with $E_{ij}\ne\varnothing$, fix an integer lift
$\widetilde a_i$ and write
$C_{ij}\equiv\widetilde a_i+p_i\tau_i\pmod{p_ip_j}$, where
$\tau_i\in\Z/p_j\Z$.
By Lemma~\ref{lem:phase}\textup{(ii)}, the admissible values of
$\tau_i$ contain a cyclic interval $T_i$ of cardinality at least
$\lfloor(2\Gamma-D)/p_i\rfloor\ge L$;
indeed, $2\Gamma-D=D+2\ceil{QL}\ge2QL$ and $p_i\le Q$. If
$E_{ij}=\varnothing$, take $T_i=\Z/p_j\Z$; its cardinality exceeds
$L$ for large $P$.

The subtuple $(p_1,\ldots,p_{j-1};p_j)$ is $R$-regular and
$|T_i|\ge L\ge(s-1)Q/R+3\ge(j-1)p_j/R+3$.
Lemma~\ref{lem:phase}\textup{(iii)}, with offsets
$\widetilde a_i$ and coefficients $p_i$, gives
$\tau_i\in T_i$ and one residue $a_j\pmod{p_j}$ such that
$a_j\equiv\widetilde a_i+p_i\tau_i\pmod{p_j}$ for every $i<j$.
This is exactly the simultaneous avoidance of all preceding forbidden
intervals. Induction on $j$ excludes every collision between distinct
host blocks; the residue choices already exclude collisions within a
single host block.

Finally, the containing interval has length
\[
 u_s+\ell_s-u_1
 =\sum_{i=1}^s\ell_i-\sum_{i=1}^{s-1}p_ip_{i+1}
  +2(s-1)\Gamma.
\]
The bounds for $\Gamma$ give the two asserted error terms.
\end{proof}

\begin{lemma}[{Hildebrand--Tenenbaum~\cite{HildebrandTenenbaum};
Montgomery--Vaughan~\cite{MV}}]
\label{lem:smooth}
Let $P^+(d)$ be the largest prime factor of $d$, with $P^+(1)=1$,
and, for $z\ge1$ and $y\ge2$, let
$\Psi(z,y)=|\{d\le z:P^+(d)\le y\}|$.
For fixed $0<\sigma<1$,
\begin{equation}\label{eq:smooth}
\Psi(z,y)\le z^\sigma\prod_{p\le y}(1-p^{-\sigma})^{-1}
\le z^\sigma\exp\!\left(O_\sigma\left(\frac{y^{1-\sigma}}{\log y}\right)\right).
\end{equation}
\end{lemma}

\section{The bounded-diameter model}\label{sec:bounded}

\begin{proof}[Proof of Theorem~\ref{thm:bounded}]
Taking $u=\sqrt n$ in Lemma~\ref{lem:bounded-lower} gives
$m_{\cP(\sqrt n)}(n)\ge \left(\frac43-o(1)\right)\frac{n^{3/2}}{\log n}$.
Since $\cP(\sqrt n)\subseteq[n]$, monotonicity gives the same lower
bound for $m(n)$.

We prove the matching upper bounds. Fix $s\ge3$, $0<\eta<1$ and
$J\in\N$, and set $x=\sqrt n$ and $\delta=(1+\eta)^{-J}$. For
$0\le j<J$, put $P_j=\delta(1+\eta)^jx$. Then
$(\delta x,x]=\bigcup_{j=0}^{J-1}(P_j,(1+\eta)P_j]$ is a disjoint
union. In each window, apply window
Lemma~\ref{lem:regular-blocks} and group all nonexceptional primes into
strongly regular $s$-tuples. For every tuple use the $A$-construction
in Lemma~\ref{lem:blocks} with $h_i=H=1$. Each exceptional prime and
each prime at most $\delta x$ is packed separately in an interval of
length $n$.

There are $O_{s,\eta}(P_j/\log P_j)$ tuples in the $j$th window. The
error per tuple in Lemma~\ref{lem:blocks} is
$O_{s,\eta}(P_j^2/R(P_j)) =O_{s,\eta}(P_j^{2-1/(s-1)}\log P_j)$,
so the sum of these errors is
$O_{s,\eta}(P_j^{3-1/(s-1)})=o(P_j^3/\log P_j)$. By
Lemma~\ref{lem:regular-blocks}, the exceptional primes also have
squared sum $o(P_j^3/\log P_j)$. The tuple saving is at least
$\alpha_{s,\eta}$ times the squared sum of the primes in that tuple.
After concatenating all host intervals, we therefore obtain
\[
m_{\cP(\sqrt n)}(n)
\le n\pi(x)-\alpha_{s,\eta}\sum_{\delta x<p\le x}p^2
+o(x^3/\log x).
\]
Since $J$ is fixed and $P_j\asymp_{\delta,\eta}x$, summing over the
$J$ windows preserves the $o(x^3/\log x)$ error term. Since $x=\sqrt n$ and
$\log x=\frac12\log n$, \eqref{eq:moments} gives
\[
 n\pi(x)=(2+o(1))\frac{n^{3/2}}{\log n},\qquad
 \sum_{\delta x<p\le x}p^2
 =\left(\frac23(1-\delta^3)+o(1)\right)
  \frac{n^{3/2}}{\log n}.
\]
Consequently,
\[
 \limsup_{n\to\infty}
 \frac{m_{\cP(\sqrt n)}(n)\log n}{n^{3/2}}
 \le2-\frac23\alpha_{s,\eta}(1-\delta^3).
\]
First let $J\to\infty$ with $s,\eta$ fixed, so $\delta\downarrow0$;
then let $s\to\infty$ and $\eta\downarrow0$. Since
$\alpha_{s,\eta}=(s-1)/(s(1+\eta)^2)\to1$, the right side tends to
$4/3$. Together with the lower bound, this proves the asserted
asymptotic for $m_{\cP(\sqrt n)}(n)$ in
\eqref{eq:full-bounded}.

For the full difference set, keep $s,\eta,J$ fixed and take $X$ and
$\cF_\delta$ as in Lemma~\ref{lem:reduction}. Thus
$X=\sqrt n(\log n)^3$ and
\[
 \cF_\delta=\bigcup_{\delta\sqrt n<p\le\sqrt n}
 \{p,2p,\ldots,\lfloor X/p\rfloor p\}.
\]
For all sufficiently large $n$, $X<n$ and $X<\delta^2n$. If an
integer $d\le X$ had two distinct prime divisors
$p,q>\delta\sqrt n$, then $d\ge pq>\delta^2n>X$, a contradiction.
Hence the displayed prime blocks are pairwise disjoint and form a
partition of $\cF_\delta$. Their sizes satisfy
$h_p=\lfloor X/p\rfloor\le H:=\lceil\delta^{-1}(\log n)^3\rceil$.
For every fixed window, $P_j\asymp_{\delta,\eta}\sqrt n$ and therefore
$\frac{H}{R(P_j)} =O_{s,\delta,\eta}\left( \frac{(\log n)^4}{n^{1/(2(s-1))}} \right)=o(1)$.
Thus Lemma~\ref{lem:blocks} applies to each strongly regular tuple
with these $h_p$, while the divisibility assertion in
Lemma~\ref{lem:regular-blocks} packs each exceptional prime block in a
separate length-$n$ interval.

For a window with lower endpoint $P$, the number of tuples is
$O(P/\log P)$ and the error per tuple is $O(HP^2/R(P))$; hence the
total construction error is
$O_{s,\eta}(HP^{3-1/(s-1)})=o(P^3/\log P)$.
The last relation follows from $H=O_\delta((\log n)^3)$ and
$P\asymp_{\delta,\eta}\sqrt n$. The exceptional squared sum is also
$o(P^3/\log P)$. Concatenating the $J$ windows and using
\eqref{eq:moments} gives
\begin{align*}
m_{\cF_\delta}(n)
&\le n\bigl(\pi(\sqrt n)-\pi(\delta\sqrt n)\bigr)
-\alpha_{s,\eta}\sum_{\delta\sqrt n<p\le\sqrt n}p^2
+o(n^{3/2}/\log n)\\
&=\left(2(1-\delta)-\frac23\alpha_{s,\eta}(1-\delta^3)+o(1)\right)
\frac{n^{3/2}}{\log n}.
\end{align*}
The sets $\cF_\delta$ and $[n]\setminus\cF_\delta$ partition $[n]$,
so concatenating their packings gives
$m(n)\le m_{\cF_\delta}(n)+m_{[n]\setminus\cF_\delta}(n)$. Adding
\eqref{eq:remainder} cancels the terms $2\delta$ and yields
\[
 \limsup_{n\to\infty}\frac{m(n)\log n}{n^{3/2}}
 \le2-\frac23\alpha_{s,\eta}(1-\delta^3).
\]
The same ordered limits $J\to\infty$, $s\to\infty$ and
$\eta\downarrow0$ give the upper constant $4/3$. Together with
$m(n)\ge m_{\cP(\sqrt n)}(n)$, this completes the proof of
\eqref{eq:full-bounded}.

We finally prove \eqref{eq:transition}. Fix $c\in(0,1)$ and assume
$c\sqrt n\le k<\sqrt n$. Since $\cP(k)\subseteq[k]$, monotonicity and
the uniform calculation following \eqref{eq:bounded-lower} give
$m_k(n)\ge \left(1-\frac{k^2}{3n}-o_c(1)\right)\frac{nk}{\log k}$.

For the upper bound, fix an integer $I\ge2$, in addition to $s$ and
$\eta$, and put $y=(\log k)^2$. Equation~\eqref{eq:smooth} with
$z=k$ and $\sigma=1/2$ gives
$\Psi(k,y)\le k^{1/2}\exp(O(\log k/\log\log k))=k^{1/2+o(1)}$.
Each $A_d$ fits by itself in a length-$n$ interval, so all differences
with $P^+(d)\le y$ have total cost
$n\Psi(k,y)=nk^{1/2+o(1)}=o(nk/\log k)$.

Every remaining $d\le k$ has a unique largest prime factor $p>y$ and
is divisible by $p$. If $y<p\le\sqrt k$, at most $\lfloor k/p\rfloor$
such differences occur. Divide them into at most
$\lceil k/p^2\rceil\le k/p^2+1$ groups of cardinality at most $p$.
By the divisibility assertion in Lemma~\ref{lem:regular-blocks}, each
group fits in length $n$. Their combined cost is at most
\[
 n\sum_{y<p\le\sqrt k}\left(\frac{k}{p^2}+1\right)
 =O\left(\frac{nk}{y}+n\sqrt k\right)=o(nk/\log k).
\]
For fixed $I$ and all sufficiently large $k$,
$y<\sqrt k<k/(I+1)$.
If $\sqrt k<p\le k/(I+1)$, all differences with largest prime factor
$p$ fit in a single length-$n$ interval, since their number is at most
$\lfloor k/p\rfloor<\sqrt k<p$.

It remains to pack the blocks with $k/(I+1)<p\le k$.
Here $h_p=\lfloor k/p\rfloor\le I$. For fixed $I$ and large $k$,
$I<p$, so every prime factor of $m\le h_p$ is smaller than $p$.
Consequently the differences whose largest prime factor is $p$ are
exactly $p,2p,\ldots,h_pp$.
To use complete geometric windows, set
$J=\lceil\log(I+1)/\log(1+\eta)\rceil$ and
$\vartheta=(I+1)^{1/J}-1\le\eta$.
The intervals $(P_j,(1+\vartheta)P_j]$, with
$P_j=k(1+\vartheta)^j/(I+1)$ and $0\le j<J$,
partition $(k/(I+1),k]$ exactly.
Use the partition into strongly regular tuples from
Lemma~\ref{lem:regular-blocks}, together with Lemma~\ref{lem:blocks}
with $\vartheta$ in place of $\eta$ and $H=I$.
Their hypotheses hold because $k<\sqrt n$ gives $n>k^2\ge p_1^2$,
and $I=o(R(P_j))$ as $k\to\infty$.
Use a separate length-$n$ interval for each exceptional block.
As above, the number of tuples times the error per tuple is
$O_{I,s,\eta}(k^{3-1/(s-1)})=o(k^3/\log k)$ over the finitely many
windows. The exceptional squared sums are also $o(k^3/\log k)$.
Since $k^2<n$, both are $o(nk/\log k)$. Also
$\vartheta\le\eta$ implies
$\alpha_{s,\vartheta}\ge\alpha_{s,\eta}$.

The length-$n$ base cost of all prime blocks is at most $n\pi(k)$;
the smooth and small-prime costs above are absorbed by
$o(nk/\log k)$. Thus
\begin{align*}
m_k(n)
&\le n\pi(k)-\alpha_{s,\eta}
\sum_{k/(I+1)<p\le k}p^2+o(nk/\log k)\\
&\le\left(1-\frac{k^2}{3n}\alpha_{s,\eta}
\bigl(1-(I+1)^{-3}\bigr)+o(1)\right)\frac{nk}{\log k}.
\end{align*}
In the last line we used, uniformly for $k\to\infty$,
\[
 n\pi(k)=(1+o(1))\frac{nk}{\log k},\qquad
 \sum_{k/(I+1)<p\le k}p^2
 =\left(\frac{1-(I+1)^{-3}}3+o(1)\right)
   \frac{k^3}{\log k}.
\]
All error terms here are uniform for $c\sqrt n\le k<\sqrt n$,
with $I,s,\eta,c$ fixed. Given $\eps>0$, choose $I,s$ sufficiently
large and then $\eta>0$ sufficiently small that
$\alpha_{s,\eta}(1-(I+1)^{-3})\ge1-\eps$.
Since $0<k^2/n<1$, this upper bound is at most
$\left(1-\frac{k^2}{3n}+\frac{\eps}{3}+o(1)\right) \frac{nk}{\log k}$.
Letting $n\to\infty$ and then $\eps\downarrow0$, and combining with
the uniform lower bound, proves \eqref{eq:transition}.
\end{proof}

\begin{proof}[Proof of Corollary~\ref{cor:bounded}]
For every $\sqrt n\le k\le n$, we have
$\cP(\sqrt n)\subseteq[k]\subseteq[n]$. Hence monotonicity gives
$m_{\cP(\sqrt n)}(n)\le m_k(n)\le m(n)$.
Both outer terms are
$(4/3+o(1))n^{3/2}/\log n$ by \eqref{eq:full-bounded}. Their error
terms do not depend on $k$, so the squeeze is uniform throughout the
stated range.
\end{proof}

\section{The equal-cardinality model}\label{sec:equal}

\begin{proof}[Proof of Theorem~\ref{thm:equal}]
Consider any packing of $(B_p)_{p\in\cP(n)}$. Its restriction to
$\cP(n-1)$ is phase-compatible, because distinct primes have greatest
common divisor $1$. Lemma~\ref{lem:compatible} gives
\begin{align*}
M_{\cP(n)}(n)
&\ge\sum_{p\le n-1}(np-p^2-p+1)\\
&=n\sum_{p\le n-1}p-\sum_{p\le n-1}p^2
  -\sum_{p\le n-1}p+\pi(n-1).
\end{align*}
By \eqref{eq:moments}, the first two sums are
$(1/2+o(1))n^3/\log n$ and $(1/3+o(1))n^3/\log n$, respectively,
while the remaining terms are $O(n^2/\log n)$. Thus
$M_{\cP(n)}(n)\ge \left(\frac16-o(1)\right)\frac{n^3}{\log n}$.

For the upper bound, fix $s\ge3$, $0<\eta<1$ and $J\in\N$, and put
$\delta=(1+\eta)^{-J}$. Partition $(\delta n,n]$ into the $J$ windows
$(P_j,(1+\eta)P_j]$, where
$P_j=\delta(1+\eta)^jn$, $0\le j<J$. In each window apply Lemma~\ref{lem:regular-blocks}, and pack every
strongly regular tuple by the $B$-family part of
Lemma~\ref{lem:blocks}. Pack each exceptional prime
and each prime at most $\delta n$ separately; $B_p$ itself fits in an
interval of length $np$.

For a window with lower endpoint $P$, there are $O(P/\log P)$ tuples,
each with error $O(P^2/R(P))$. Their total error is
$O_{s,\eta}(P^{3-1/(s-1)})=o(P^3/\log P)$. The exceptional squared sum is also $o(P^3/\log P)$. Since $J$ is fixed and
$P_j\asymp_{\delta,\eta}n$, the total error is $o(n^3/\log n)$.
Concatenation and the tuple savings give
\begin{align*}
M_{\cP(n)}(n)
&\le n\sum_{p\le n}p-\alpha_{s,\eta}\sum_{\delta n<p\le n}p^2
+o(n^3/\log n)\\
&=\left(\frac12-\frac{\alpha_{s,\eta}}3(1-\delta^3)+o(1)\right)
\frac{n^3}{\log n}.
\end{align*}
First take $n\to\infty$ with $s,\eta,J$ fixed; then let
$J\to\infty$, so $\delta\downarrow0$; finally let $s\to\infty$ and
$\eta\downarrow0$. Since $\alpha_{s,\eta}\to1$, the upper coefficient tends to
$1/2-1/3=1/6$, proving the asserted asymptotic for
$M_{\cP(n)}(n)$.

We next include composite differences through compatible subfamilies.
With the odd-prime sums $S_j$ defined in Section~\ref{sec:prelim},
we first establish the finite inequality
\begin{equation}\label{eq:parity}
M_k(n)\ge(n-2)S_1(K)-S_2(K)+\frac n2S_1(x)-\frac32S_2(x)
\end{equation}
for every integer $2\le K\le\min\{k,n-1\}$ and every real
$0\le x\le K/2$.

Fix an arbitrary packing $(\sigma_d+B_d)_{d\le k}$ into $[N]$. For
$\theta\in\{0,1\}$ and each odd prime $p\le K$, define
\[
 d_p^{(\theta)}=
 \begin{cases}
 2p,&p\le x\text{ and }\sigma_{2p}\equiv\theta\pmod2,\\
 p,&\text{otherwise}.
 \end{cases}
\]
If $p\le x$, then $2p\le K$; hence every selected difference lies in
$[K]\subseteq[n-1]$. The differences are distinct: an odd prime cannot
equal twice an odd prime, and $2p=2q$ implies $p=q$.

Let $D_\theta=\{d_p^{(\theta)}:p\le K, p\text{ odd}\}$. For distinct
odd primes $p,q$, the gcd of the selected differences is $1$, unless
both are $2p,2q$. In that case the gcd is $2$, and the corresponding
shifts both have parity $\theta$. Thus $D_\theta$ is phase-compatible.
With $F(d)=nd-d^2$, averaging over the two choices gives
\begin{align*}
\frac12\sum_{\theta=0}^1\sum_{d\in D_\theta}F(d)
&=\sum_{p\le K,\ p\text{ odd}}F(p)
+\frac12\sum_{p\le x,\ p\text{ odd}}\bigl(F(2p)-F(p)\bigr)\\
&=nS_1(K)-S_2(K)+\frac n2S_1(x)-\frac32S_2(x).
\end{align*}
For $p>x$ both choices contribute $F(p)$; for $p\le x$ exactly one
choice contributes $F(2p)$ and the other contributes $F(p)$. The last
equality follows from $F(2p)-F(p)=np-3p^2$.
Choose $\theta$ for which $\sum_{d\in D_\theta}F(d)$ is at least the
displayed average. Since $\sum_{d\in D_\theta}d\le2S_1(K)$,
Lemma~\ref{lem:compatible} gives
\begin{align*}
N&\ge\sum_{d\in D_\theta}(nd-d^2-d+1)\\
&\ge\sum_{d\in D_\theta}F(d)-2S_1(K)\\
&\ge(n-2)S_1(K)-S_2(K)+\frac n2S_1(x)-\frac32S_2(x).
\end{align*}
Since the original packing was arbitrary, this proves
\eqref{eq:parity}.

If $k\to\infty$ and $k=o(n)$, take $K=k$ and $x=k/2$.
For large $n$, $k\le n-1$, so this choice is admissible. By
\eqref{eq:moments},
\[
 S_1(k)=\left(\frac12+o(1)\right)\frac{k^2}{\log k},
 \qquad
 S_1(k/2)=\left(\frac18+o(1)\right)\frac{k^2}{\log k}.
\]
Moreover, $S_2(k),S_2(k/2)=O(k^3/\log k)$ and
$S_1(k)=O(k^2/\log k)$. Since $k/n\to0$, the negative terms
$2S_1(k)+S_2(k)+(3/2)S_2(k/2)$ are $o(nk^2/\log k)$. Hence
\eqref{eq:parity} gives
\[
 M_k(n)\ge
 \left(\frac12+\frac1{16}-o(1)\right)\frac{nk^2}{\log k}
 =\left(\frac9{16}-o(1)\right)\frac{nk^2}{\log k}.
\]

If $k/n\to\lambda\in(0,1)$, take $K=k$ and
$x=\min\{k/2,n/3\}$, so $x/n\to a=\min\{\lambda/2,1/3\}$.
Again $k\le n-1$ eventually. Since
$\log k\sim\log x\sim\log n$, \eqref{eq:moments} gives
\begin{align*}
S_1(k)&=\left(\frac{\lambda^2}{2}+o(1)\right)
       \frac{n^2}{\log n},&
S_2(k)&=\left(\frac{\lambda^3}{3}+o(1)\right)
       \frac{n^3}{\log n},\\
S_1(x)&=\left(\frac{a^2}{2}+o(1)\right)
       \frac{n^2}{\log n},&
S_2(x)&=\left(\frac{a^3}{3}+o(1)\right)
       \frac{n^3}{\log n}.
\end{align*}
The term $2S_1(k)=O(n^2/\log n)$ is $o(n^3/\log n)$.
Substitution in \eqref{eq:parity} yields
\[
 M_k(n)\ge\left(
 \frac{\lambda^2}{2}-\frac{\lambda^3}{3}
 +\frac{a^2}{4}-\frac{a^3}{2}-o(1)
 \right)\frac{n^3}{\log n}.
\]
Finally, $9/16>1/2$, and, for $0<a\le1/3$,
$a^2/4-a^3/2=a^2(1-2a)/4>0$.
Thus both predictions for $M_k(n)$ in
Conjecture~\ref{conj:equal-main} are false.
\end{proof}

\begin{proof}[Proof of Corollary~\ref{cor:equal}]
Fix $\lambda\in[2/3,1)$ and put $k_n=\lfloor\lambda n\rfloor$.
Then $k_n/n\to\lambda$ and $[k_n]\subseteq[n]$, so
$M(n)=M_n(n)\ge M_{k_n}(n)$. In Theorem~\ref{thm:equal},
$a=\min\{\lambda/2,1/3\}=1/3$. Hence
\[
 \liminf_{n\to\infty}\frac{M(n)\log n}{n^3}
 \ge\frac{\lambda^2}{2}-\frac{\lambda^3}{3}
   +\frac1{36}-\frac1{54}
 =\frac{\lambda^2}{2}-\frac{\lambda^3}{3}+\frac1{108}.
\]
This inequality holds for every fixed $\lambda<1$. Letting
$\lambda\uparrow1$ after taking the limit inferior gives
\[
 \liminf_{n\to\infty}\frac{M(n)\log n}{n^3}
 \ge\frac16+\frac1{108}=\frac{19}{108}.
\]
Since $19/108>1/6$, this also disproves the full-family assertion in
Conjecture~\ref{conj:equal-main}. No uniformity in $\lambda$ is
required.
\end{proof}

The sharp leading constants for $M_k(n)$ remain undetermined in these
regimes. The bound \eqref{eq:parity} uses only the pairs $p,2p$;
compatible selections among further small multiples of primes may
give stronger bounds.

\section*{Declaration on the use of AI}
The authors used generative AI tools to assist in discussing proof strategies, checking proofs, and improving exposition. All mathematical arguments, results, and conclusions were reviewed and verified by the authors.

\end{document}